\documentclass[11pt]{amsart}

\usepackage{amsmath,amssymb,tikz,graphicx, amsfonts,stmaryrd,hyperref}

\usetikzlibrary{decorations.markings,arrows.meta,calc, positioning,calc}
\tikzstyle{vertex}=[circle, fill=black, minimum size=5pt, inner sep=0]

\theoremstyle{plain}
\newtheorem{thm}{Theorem}[section]
\newtheorem{lem}[thm]{Lemma}

\newtheorem{cor}[thm]{Corollary}
\newtheorem{obs}[thm]{Observation}

\theoremstyle{definition}

\theoremstyle{remark}

\DeclareMathOperator{\CN}{\nu}
\DeclareMathOperator{\Lo}{\mathsf L}
\DeclareMathOperator{\N}{\mathsf N}
\DeclareMathOperator{\Z2}{\mathbb Z_2}
\DeclareMathOperator{\ind}{ind_{\mathbb Z_2}}
\DeclareMathOperator{\coind}{coind_{\mathbb Z_2}}
\DeclareMathOperator{\cohom-ind}{cohom-ind_{\mathbb Z_2}}

\DeclareMathOperator{\OS}{\mathbb S}
\DeclareMathOperator{\NS}{\mathbb N}
\DeclareMathOperator{\sd}{\mathsf{sd}}

\begin{document}
\title{Quadrangulations and the Lovász complex}
\date{October 4, 2025}
\author{Carmen Arana}
\author{Mat\v ej Stehl\'ik}
\address{Université Paris Cité, CNRS, IRIF, F-75006, Paris, France.}
\email{\{arana,matej\}@irif.fr}
\thanks{The first author was funded during her master's by the PGSM program. The second author was supported by an Émergence en recherche IDEX grant from Université Paris Cité}

\begin{abstract}
    The Lovász complex $\Lo(G)$ of a graph $G$ is a deformation retract of its neighborhood complex, equipped with a canonical $\Z2$-action. We show that, under mild assumptions, $\Lo(G)$ is homeomorphic to a surface if and only if $G$ is a non-bipartite quadrangulation of the orbit space $\Lo(G)/\Z2$ in which every $4$-cycle is facial. This yields a classification of the Lovász complexes of all such quadrangulations. As an application, we contextualize a result of Archdeacon \emph{et al.}\ and Mohar and Seymour on the chromatic number of quadrangulations, obtaining a stronger statement about the $\Z2$-index.
\end{abstract}
    
\maketitle
    
\section{Introduction}

In his proof of Kneser's conjecture, Lovász~\cite{Lov78} introduced the neighborhood complex $\N(G)$ of a graph $G$ and showed that its homotopic connectivity gives a lower bound on the chromatic number of $G$. In the same paper he also defined another, lesser-known simplicial complex, now called the Lovász complex $\Lo(G)$ (see, e.g.,~\cite{dL13, Mat03}). The two complexes are homotopy equivalent---indeed, $\Lo(G)$ is a deformation retract of $\N(G)$~\cite{Cso08}. Although $\Lo(G)$ typically has more vertices than $\N(G)$, it has two advantages: it admits a canonical free $\Z2$-action, essential for applications of the Borsuk--Ulam theorem, and can sometimes have lower dimension, which is crucial for our purposes.

From a graph-theoretic perspective, these complexes may appear somewhat mysterious. Indeed, Toft~\cite{Tof95} remarked that Lovász's topological bound's ``graph theoretic meaning is not well understood.'' Our goal is to shed light on the Lovász complex in what is arguably the simplest nontrivial case---when it forms a surface. We show that the Lovász complex is closely related to \emph{quadrangulations}---graphs embedded in surfaces such that every face is bounded by four edges. More precisely, we prove the following.

\begin{thm} \label{thm:L-quad}
    If $G \not\cong K_{2,3}$ is a connected, non-bipartite quadrangulation of a surface $S$ in which every $4$-cycle is facial, then $\Lo(G)$ is a double cover of $S$. Conversely, if $\Lo(G)$ is a closed surface, no neighborhood of $G$ dominates another, and $G$ is $K_{2,3}$-free, then $G$ is a quadrangulation of $\Lo(G)/\Z2$ in which all $4$-cycles bound a face.
\end{thm}

As a corollary of our main result, we classify all Lovász complexes of quadrangulations of surfaces (subject to mild assumptions, see Corollary~\ref{cor:class_L}) and, conversely, all graphs whose Lovász complex is a surface (also subject to mild assumptions, see Corollary~\ref{cor:class_graphs}).

The chromatic number, denoted $\chi$, of surface quadrangulations has been studied extensively. Hutchinson~\cite{Hut95} showed that all quadrangulations of orientable surfaces with sufficiently large edge-width are $3$-colorable. The non-orientable case exhibits different behavior: Youngs~\cite{You96} proved that every quadrangulation of the projective plane is either bipartite or $4$-chromatic. This result was later extended to other non-orientable surfaces by Archdeacon \emph{et al.}~\cite{AHNNO01}, and independently by Mohar and Seymour~\cite{MS02}. A quadrangulation $G$ of a non-orientable surface $S$ is \emph{odd} if it contains an odd cycle $C$ such that cutting $S$ along $C$ results in an orientable surface.

\begin{thm}[\cite{AHNNO01,MS02}]\label{thm:archdeacon}
    If $G$ is an odd quadrangulation of a non-orientable surface, then $\chi(G) \geq 4$.
\end{thm}

The proof in~\cite{AHNNO01,MS02} uses the \emph{winding number}, so it is clearly topological, yet it is not obvious how it relates to the theory of \emph{topological methods} originating in Lovász's proof of Kneser's conjecture~\cite{Lov78}; see~\cite{dL13,Koz08,Mat03,MZ04}. A modern interpretation of Lovász's bound uses the \emph{$\Z2$-index} (definitions are given in Section~\ref{sec:preliminaries}).

\begin{thm}[\cite{dL13,Koz08,Mat03,MZ04}]\label{thm:lovasz}
    If $G$ is a (non-empty) graph, then $\chi(G) \geq \ind(\Lo(G))+2$.
\end{thm}

In the final part of our paper, we prove the following strengthening of Theorem~\ref{thm:archdeacon}.
\begin{thm}\label{thm:archdeacon-ind}
    If $G$ is an odd quadrangulation of a non-orientable surface, then $\ind(\Lo(G)) \geq 2$.
\end{thm}

A weaker topological bound on the chromatic number can be obtained by replacing the $\Z2$-index in Theorem~\ref{thm:lovasz} by the \emph{$\Z2$-coindex} (see Section~\ref{sec:preliminaries} for a definition). Spaces where these two invariants do not coincide were called \emph{non-tidy} spaces by Matoušek~\cite[pp.~100--101]{Mat03}. While examples of non-tidy spaces are known (see, e.g.~\cite{Cso05,Mat03}), we are unaware of explicit examples of graphs whose Lovász complex is non-tidy.

Using the universal cover, one can show that $\coind(S) = 1$ for any closed surface $S$ other than the sphere; see, e.g.,~\cite[p.~100]{Mat03}. Therefore, by Theorem~\ref{thm:L-quad}, if $G$ is an odd quadrangulation of a non-orientable surface other than the projective plane, with all $4$-cycles facial, then $\coind(\Lo(G)) = 1$. On the other hand, Theorem~\ref{thm:archdeacon-ind} implies that $\ind(\Lo(G)) = 2$. We conclude that $\Lo(G)$ is non-tidy for such graphs.

\section{Preliminaries}
\label{sec:preliminaries}

A \emph{graph} $G$ is a pair of sets $V(G)$ and $E(G)$, the set of \emph{vertices} and \emph{edges} of $G$, respectively. If $G$ and $H$ are isomorphic graphs, we write $G \cong H$. We shall assume throughout that $V(G)$ is finite. We will say that a graph is \emph{$K_{2,3}$-free} if it does not contain the complete bipartite graph $K_{2,3}$ as a subgraph. A \emph{$k$-coloring} of a graph $G$ is a function $c:V(G) \to \{1, \ldots, k\}$ such that $c(u) \neq c(v)$ for every edge $\{u,v\} \in E(G)$. The \emph{chromatic number} of $G$ is the smallest $k$ for which a $k$-coloring of $G$ exists. Let $G$ be a graph, $v \in V(G)$ a vertex, and $A \subseteq V(G)$ a set of vertices. The \emph{neighborhood} of $v$ is denoted by $N(v)$. We will say that the neighborhood of $v$ \emph{dominates} the neighborhood of $u$ if $N(u) \subseteq N(v)$. The set of \emph{common neighbors} of $A$ is
\[
    \CN(A) := \{v \in V(G) \mid \{u, v\} \in E(G) \text{ for all } u \in A\}.
\]
For more background on graph-theoretic notions, see Bondy and Murty~\cite{BM08}.

If $X$ and $Y$ are homeomorphic topological spaces, we write $X \cong Y$. A \emph{double cover} of a topological space $X$ is a $\Z2$-space $\widetilde{X}$ such that the orbit space $\widetilde{X} / \Z2$ is homeomorphic to $X$. For $k \geq 1$, a topological space $X$ is \emph{$k$-connected} if the homotopy groups $\pi_1(X)$, \ldots , $\pi_k(X)$ are all trivial. A \emph{surface} is a connected compact Hausdorff topological space $S$ locally homeomorphic to an open disc. According to the classification of surfaces, a closed surface is uniquely determined, up to homeomorphism, by its Euler characteristic and orientability. We denote the orientable surface of genus $k$ by $\OS_k$, and the non-orientable surface of genus $k$ by $\NS_k$. Note that $\OS_k$ and $\NS_k$ have Euler characteristic $2-2k$ and $2-k$, respectively. For further topological background, see Hatcher~\cite{Hat02}.

A \emph{quadrangulation} of a surface $S$ is a graph embedded in $S$ such that the boundary of each face is a $4$-cycle. Following~\cite{AHNNO01}, a quadrangulation of a non-orientable surface is called \emph{odd} if there exists an odd cycle in the graph whose removal (i.e., cutting the surface along the cycle) yields an orientable surface. A cycle in an embedded graph is \emph{one-sided} if a small strip around it forms a Möbius strip. For more background on graphs embedded in surfaces, see Mohar and Thomassen~\cite{MT01}.

A \emph{$\Z2$-space} is a topological space equipped with a free $\Z2$-action. A canonical example of a $\Z2$-space is the $n$-dimensional sphere $\mathbb S^n$ equipped with the antipodal action. Given $\Z2$-spaces $X$ and $Y$, we write $X \xrightarrow{\Z2} Y$ if there exists a continuous $\Z2$-equivariant map from $X$ to $Y$. We define the following invariants:
\begin{align*}
    \ind(X)   &= \min \{n \in \mathbb{N} \mid X \xrightarrow{\Z2} \mathbb S^n\}, \\
    \coind(X) &= \max \{n \in \mathbb{N} \mid \mathbb S^n \xrightarrow{\Z2} X\}.
\end{align*}
Note that the Borsuk--Ulam theorem can be succinctly expressed by the inequality $\ind(\OS^n) \geq n$. Generalizations of the Borsuk--Ulam theorem to manifolds other than the sphere---also expressible using the $\Z2$-index---have been studied in~\cite{BM22,Mus12,GHZ10}. We will also make use of the \emph{cohomological index} (also known as the \emph{Stiefel--Whitney height}), denoted $\cohom-ind(X)$, defined as the height of the first Stiefel–Whitney class of $X$. The following inequalities hold for any $k$-connected $\Z2$-space $X$:
\begin{equation}\label{eq:inequalities}
    k+1 \leq \coind(X) \leq \cohom-ind(X) \leq \ind(X).
\end{equation}
See Kozlov~\cite{Koz08} or Matoušek~\cite{Mat03} for more details.

An (abstract) \emph{simplicial complex} $\mathsf K$ is a finite hereditary set system. Its vertex set is denoted by $V(\mathsf K)$, and its barycentric subdivision by $\sd(\mathsf K)$. Any abstract simplicial complex $\mathsf K$ can be realized as a topological space $|\mathsf K|$ in $\mathbb R^d$ for some $d$. The \emph{neighborhood complex} of $G$ is the simplicial complex
\[
    \N(G) := \{A \subseteq V(G) \mid \CN(A) \neq \emptyset\}.
\]
\begin{figure} 
    \centering
    \begin{tikzpicture}[scale=1.2,every edge/.append style={semithick}]
        \node[vertex,label=above:\small $1$] (1) at (0.33,-0.33) {};
        \node[vertex,label=below right:\small $2$] (2) at (1,-1) {};
        \node[vertex,label=below left:\small $3$] (3) at (-1,-1) {};
        \node[vertex,label=above:\small $4$] (4) at (-0.33,0.33) {};
        \node[vertex,label=above right:\small $5$] (5) at (1,1) {};
        \node[vertex,label=above left:\small $6$] (6) at (-1,1) {};
        \draw (1) edge (2)
              (1) edge (3)
              (1) edge (4)
              (2) edge (3)
              (2) edge (5)
              (3) edge (4)
              (3) edge (6)
              (4) edge (5)
              (5) edge (6);
    \end{tikzpicture}
    \hfil 
    \begin{tikzpicture}[scale=1.2,every edge/.append style={semithick}]
        \foreach \i in {1,2,...,8}{
            \coordinate (p\i) at (45+45*\i:1);
        }
        \fill[fill=gray,line cap=round] (p2) -- (p4) -- (p3) -- cycle;
        \fill[fill=gray,line cap=round] (p6) -- (p8) -- (p7) -- cycle;
        \draw[semithick] (p1)
        \foreach \i in {2,3,...,8}{
            -- (p\i)
        }
        -- cycle;
        \draw (p2) edge (p4);
        \draw (p6) edge (p8);
        \node at ($(p1)+(0,0.2)$) {\small $\{1\}$};
        \node at ($(p2)+(-0.6,0)$) {\small $\{1,3,5\}$};
        \node at ($(p3)+(-0.5,0)$) {\small $\{3,5\}$};
        \node at ($(p4)+(-0.3,0)$) {\small $\{3\}$};
        \node at ($(p5)+(0,-0.2)$) {\small $\{2,3,4\}$};
        \node at ($(p6)+(0.5,0)$) {\small $\{2,4\}$};
        \node at ($(p7)+(0.6,0)$) {\small $\{2,4,6\}$};
        \node at ($(p8)+(0.8,0)$) {\small $\{1,2,4,6\}$};
    \end{tikzpicture}
    \caption{A graph (left) and its Lovász complex (right).}
    \label{Lov}
\end{figure}
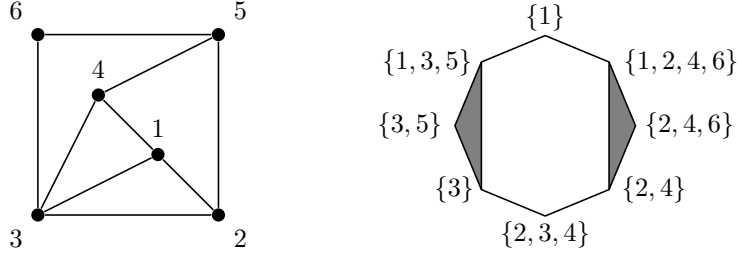
The \emph{Lovász complex} of $G$, denoted $\Lo(G)$ (see Figure~\ref{Lov}), is the subcomplex of the barycentric subdivision $\sd(\N(G))$ induced by the vertex set
\[
    V(\Lo(G)) = \{A \subseteq V(G) \mid \CN^2(A) = A\}.
\]
For more on topological methods in combinatorics, see Matoušek~\cite{Mat03}, de Longueville~\cite{dL13}, and Kozlov~\cite{Koz08}.

\section{Lovász complex of quadrangulations} \label{ID} 

The goal of this section is to prove the ``forward'' direction of Theorem~\ref{thm:L-quad}. We begin with two simple lemmas.

\begin{lem}\label{order of hyp}
    Let $G$ be a connected non-bipartite quadrangulation of a surface in which every $4$-cycle is facial. Then either $G \cong K_{2,3}$, or $G$ is $K_{2,3}$-free and no neighborhood dominates another.
\end{lem}
    
\begin{proof}
    Suppose $G$ contains $K_{2,3}$ with shores $\{a_1,a_2\}$ and $\{b_1,b_2,b_3\}$. Then the cycles $\{a_1,b_1,a_2,b_2\}$,$\{a_1,b_2,a_2,b_3\}$ and $\{a_1,b_3,a_2,b_1\}$ must all bound faces. Each edge $\{a_i,b_j\}$ is already in two faces, so cannot be in any other faces. Therefore, $G \cong K_{2,3}$.

    Now suppose $G$ is $K_{2,3}$-free. Let $a,b$ be two distinct vertices of $G$ such that $N(b) \subseteq N(a)$. As $G$ is a quadrangulation, $|N(b)| \geq 2$. If $|N(b)|\geq 3$, then $|\CN(\{a,b\})|=|N(b)|\geq 3$, so $G$ contains a $4$-cycle which does not bound a face.

    It remains to consider the case $|N(b)|=2$, so suppose $N(b) = \{c,d\}$. As $G$ is non-bipartite, $G$ is not a $4$-cycle. The vertex $b$ has only $c,d$ as neighbors, but the edge $\{c,b\}$ is in two faces by definition. The second face must also include $\{b,d\}$ because $b$ is not incident to other edges, and each face must have two edges incident to each vertex. Such a face has one more vertex $e \neq a$ such that $\{c,e\}$ and $\{d,e\}$ are edges of $G$. But that implies $|\CN(\{c,d\})| \geq 3$, so we again conclude that $G$ contains a $4$-cycle which does not bound a face.
\end{proof}

\begin{lem}\label{lem:all_vertices}
    Let $G \not\cong K_{2,3}$ be a connected non-bipartite quadrangulation of a surface in which every $4$-cycle is facial. Then $A \in V(\Lo(G))$ if and only if $A$ satisfies one of the following conditions:
    \begin{enumerate}
        \item $A$ is a pair of opposite vertices in a face of $G$;\label{lem:all_vertices:1}
        \item $A = N(v)$, for some vertex $v \in V(G)$;\label{lem:all_vertices:2}
        \item $A = \{v\}$, for some vertex $v \in V(G)$.\label{lem:all_vertices:3}
    \end{enumerate}
\end{lem}
    
\begin{proof}
    For the sake of readability, we call $A \in V(\Lo(G))$ a \emph{diagonal}, \emph{singleton}, or \emph{neighborhood} if $A$ satisfies condition (\ref{lem:all_vertices:1}), (\ref{lem:all_vertices:2}) or (\ref{lem:all_vertices:3}), respectively.

    We will first show that if $A \subseteq V(G)$ satisfies one of the conditions (\ref{lem:all_vertices:1})--(\ref{lem:all_vertices:3}), then $A \in V(\Lo(G))$.
    Since no neighborhood dominates another, $\CN(N(a)) = \{a\}$. Hence, $\{a\} \in V(\Lo(G))$ and $N(a) \in V(\Lo(G))$ for any $a \in V(G)$.

    Now consider a $4$-cycle in $G$ with vertices $a,b,c,d$ (in cyclic order), and set $A = \{a,c\}$. Then $\CN(A) = \{b,d\}$ and $\CN(\{b,d\}) = A$ by the hypothesis. Therefore,
    \[
        \CN^2(A) = \CN(\{b,d\}) = A,
    \]
    which proves that $A \in V(\Lo(G))$.

    It remains to show that if $A \in V(\Lo(G))$, then $A$ satisfies one of the conditions (\ref{lem:all_vertices:1})--(\ref{lem:all_vertices:3}). Suppose $A \in V(\Lo(G))$ and $A$ does not satisfy the conditions (\ref{lem:all_vertices:1}) and (\ref{lem:all_vertices:3}). Then $|A| \geq 2$, and $A$ is not a diagonal, so $A$ has only one common neighbor $w$. So $N(w)=\CN(\{w\})=\CN^2(A)=A$. We conclude that $A$ satisfies condition (\ref{lem:all_vertices:2}).
\end{proof}

\begin{figure}
    \centering
    \begin{tikzpicture}[scale=1.4,every edge/.append style={semithick}]
        \foreach \i in {1,2,...,4}{
            \node[vertex] (p\i) at (45+90*\i:1) {};
        }
        \draw[semithick] (p1) -- (p2) -- (p3) -- (p4) -- (p1);
        \node at ($(p1)+(-0.2,0.2)$) {\small $1$};
        \node at ($(p2)+(-0.2,-0.2)$) {\small $2$};
        \node at ($(p3)+(0.2,-0.2)$) {\small $3$};
        \node at ($(p4)+(0.2,0.2)$) {\small $4$};
    
        \begin{scope}[shift={(3.5,0)}]
            \foreach \i in {1,2,...,4}{
                \node[vertex] (p\i) at (45+90*\i:1) {};
            }
            \node[vertex] (p0) at (0,0) {};
            \draw[semithick] (p1) -- (p2) -- (p3) -- (p4) -- (p1);
            \foreach \i in {1,2,...,4}{
                \draw[semithick] (p0) -- (p\i);
            }
            \node at ($(p0)+(0.4,0)$) {\small $\{1,3\}$};
            \node at ($(p1)+(-0.2,0.2)$) {\small $\{1\}$};
            \node at ($(p2)+(-0.2,-0.2)$) {\small $N(2)$};
            \node at ($(p3)+(0.2,-0.2)$) {\small $\{3\}$};
            \node at ($(p4)+(0.2,0.2)$) {\small $N(4)$};
        \end{scope}
    
        \begin{scope}[shift={(6,0)}]
            \foreach \i in {1,2,...,4}{
                \node[vertex] (p\i) at (45+90*\i:1) {};
            }
            \node[vertex] (p0) at (0,0) {};
            \draw[semithick] (p1) -- (p2) -- (p3) -- (p4) -- (p1);
            \foreach \i in {1,2,...,4}{
                \draw[semithick] (p0) -- (p\i);
            }
            \node at ($(p0)+(0.4,0)$) {\small $\{2,4\}$};
            \node at ($(p1)+(-0.2,0.2)$) {\small $N(1)$};
            \node at ($(p2)+(-0.2,-0.2)$) {\small $\{2\}$};
            \node at ($(p3)+(0.2,-0.2)$) {\small $N(3)$};
            \node at ($(p4)+(0.2,0.2)$) {\small $\{4\}$};
        \end{scope}
    \end{tikzpicture}
    \caption{A face of a quadrangulation (left) and the corresponding faces of the Lovász complex (right).}
    \label{faces_Lovasz}
\end{figure}
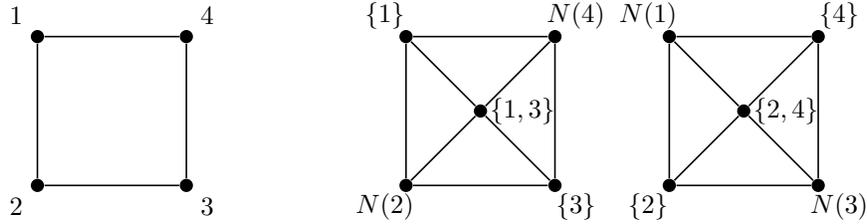

We are now ready to prove the ``forward'' direction of Theorem~\ref{thm:L-quad}.

\begin{thm} \label{thm:dc_lovasz}
    Let $G \not\cong K_{2,3}$ be a connected non-bipartite quadrangulation of a surface $S$ in which every $4$-cycle is facial. Then $\Lo(G)$ is homeomorphic to a double cover of $S$.
\end{thm}

\begin{proof}
    Lemma~\ref{lem:all_vertices} implies that the maximal simplices of $\Lo(G)$ are triangles. Each face of $G$ therefore induces eight faces in $\Lo(G)$ (see Figure~\ref{faces_Lovasz}). It follows that $\Lo(G)$ is a triangulation of a double cover of $S$. Indeed, the link of every vertex of $\Lo(G)$ is a cycle: the link of a diagonal is a $4$-cycle, while the links of a singleton ${v}$ and of a neighborhood $N(v)$ are isomorphic to subdivisions of the link of $v$ in $G$, which is a cycle.
\end{proof}

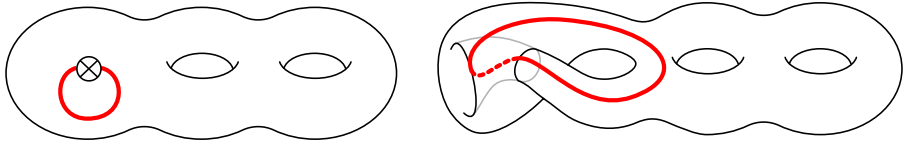
\begin{figure}
    \centering
    \begin{tikzpicture}[scale=0.27]
        \draw[black,semithick] 
        (171.5352pt, -497.9958pt) -- (195.9884pt, -521.2158pt)
        ;
        \draw[black,semithick]  
        (195.9884pt, -497.9958pt) -- (171.5352pt, -521.2158pt)
        ;
        \draw[red,ultra thick] 
        (166.3682pt, -508.3833pt) .. controls (132.643pt, -517.088pt) and (135.5687pt, -579.5563pt) .. (181.7329pt, -579.5762pt)
         -- (181.7329pt, -579.5762pt)
         -- (181.7329pt, -579.5762pt)
         -- (181.7329pt, -579.5762pt) .. controls (239.5314pt, -579.601pt) and (234.3977pt, -514.1229pt) .. (200.6331pt, -508.6017pt)
        ;
        \draw[black,semithick,line cap=round,line join=round]
        (283.7427pt, -436.1364pt) .. controls (269.6824pt, -442.7925pt) and (253.4542pt, -442.7925pt) .. (239.394pt, -436.1364pt)
         -- (239.394pt, -436.1364pt)
         -- (239.394pt, -436.1364pt)
         -- (239.394pt, -436.1364pt) .. controls (222.5034pt, -428.1403pt) and (202.8476pt, -423.5565pt) .. (181.8672pt, -423.5565pt)
         -- (181.8672pt, -423.5565pt)
         -- (181.8672pt, -423.5565pt)
         -- (181.8672pt, -423.5565pt) .. controls (118.8286pt, -423.5565pt) and (67.72566pt, -464.9195pt) .. (67.72566pt, -515.943pt)
         -- (67.72566pt, -515.943pt)
         -- (67.72566pt, -515.943pt)
         -- (67.72566pt, -515.943pt) .. controls (67.72566pt, -566.9666pt) and (118.8286pt, -608.3298pt) .. (181.8672pt, -608.3298pt)
         -- (181.8672pt, -608.3298pt)
         -- (181.8672pt, -608.3298pt)
         -- (181.8672pt, -608.3298pt) .. controls (202.8476pt, -608.3298pt) and (222.5034pt, -603.7457pt) .. (239.394pt, -595.7496pt)
         -- (239.394pt, -595.7496pt)
         -- (239.394pt, -595.7496pt)
         -- (239.394pt, -595.7496pt) .. controls (253.4542pt, -589.0935pt) and (269.6824pt, -589.0935pt) .. (283.7427pt, -595.7496pt)
         -- (283.7427pt, -595.7496pt)
         -- (283.7427pt, -595.7496pt)
         -- (283.7427pt, -595.7496pt) .. controls (300.6334pt, -603.7457pt) and (320.2891pt, -608.3298pt) .. (341.2627pt, -608.3298pt)
         -- (341.2627pt, -608.3298pt)
         -- (341.2627pt, -608.3298pt)
         -- (341.2627pt, -608.3298pt) .. controls (362.2363pt, -608.3298pt) and (381.892pt, -603.7457pt) .. (398.7827pt, -595.7496pt)
         -- (398.7827pt, -595.7496pt)
         -- (398.7827pt, -595.7496pt)
         -- (398.7827pt, -595.7496pt) .. controls (412.8429pt, -589.0935pt) and (429.0711pt, -589.0935pt) .. (443.1313pt, -595.7496pt)
         -- (443.1313pt, -595.7496pt)
         -- (443.1313pt, -595.7496pt)
         -- (443.1313pt, -595.7496pt) .. controls (460.0219pt, -603.7457pt) and (479.6777pt, -608.3298pt) .. (500.6581pt, -608.3298pt)
         -- (500.6581pt, -608.3298pt)
         -- (500.6581pt, -608.3298pt)
         -- (500.6581pt, -608.3298pt) .. controls (563.6968pt, -608.3298pt) and (614.7997pt, -566.9666pt) .. (614.7997pt, -515.943pt)
         -- (614.7997pt, -515.943pt)
         -- (614.7997pt, -515.943pt)
         -- (614.7997pt, -515.943pt) .. controls (614.7997pt, -464.9195pt) and (563.6968pt, -423.5565pt) .. (500.6581pt, -423.5565pt)
         -- (500.6581pt, -423.5565pt)
         -- (500.6581pt, -423.5565pt)
         -- (500.6581pt, -423.5565pt) .. controls (479.6777pt, -423.5565pt) and (460.0219pt, -428.1403pt) .. (443.1313pt, -436.1364pt)
         -- (443.1313pt, -436.1364pt)
         -- (443.1313pt, -436.1364pt)
         -- (443.1313pt, -436.1364pt) .. controls (429.0711pt, -442.7925pt) and (412.8429pt, -442.7925pt) .. (398.7827pt, -436.1364pt)
         -- (398.7827pt, -436.1364pt)
         -- (398.7827pt, -436.1364pt)
         -- (398.7827pt, -436.1364pt) .. controls (381.892pt, -428.1403pt) and (362.2363pt, -423.5565pt) .. (341.2627pt, -423.5565pt)
         -- (341.2627pt, -423.5565pt)
         -- (341.2627pt, -423.5565pt)
         -- (341.2627pt, -423.5565pt) .. controls (320.2891pt, -423.5565pt) and (300.6334pt, -428.1403pt) .. (283.7427pt, -436.1364pt) -- cycle
        ;
        \draw[black,semithick] 
        (200.734pt, -509.7149pt) .. controls (200.734pt, -505.0592pt) and (199.088pt, -501.0853pt) .. (195.7959pt, -497.7932pt)
         -- (195.7959pt, -497.7932pt)
         -- (195.7959pt, -497.7932pt)
         -- (195.7959pt, -497.7932pt) .. controls (192.5038pt, -494.5012pt) and (188.5299pt, -492.8552pt) .. (183.8742pt, -492.8551pt)
         -- (183.8742pt, -492.8551pt)
         -- (183.8742pt, -492.8551pt)
         -- (183.8742pt, -492.8551pt) .. controls (179.2185pt, -492.8552pt) and (175.2446pt, -494.5012pt) .. (171.9525pt, -497.7932pt)
         -- (171.9525pt, -497.7932pt)
         -- (171.9525pt, -497.7932pt)
         -- (171.9525pt, -497.7932pt) .. controls (168.6604pt, -501.0853pt) and (167.0144pt, -505.0592pt) .. (167.0144pt, -509.7149pt)
         -- (167.0144pt, -509.7149pt)
         -- (167.0144pt, -509.7149pt)
         -- (167.0144pt, -509.7149pt) .. controls (167.0144pt, -514.3707pt) and (168.6604pt, -518.3446pt) .. (171.9525pt, -521.6366pt)
         -- (171.9525pt, -521.6366pt)
         -- (171.9525pt, -521.6366pt)
         -- (171.9525pt, -521.6366pt) .. controls (175.2446pt, -524.9287pt) and (179.2185pt, -526.5748pt) .. (183.8742pt, -526.5747pt)
         -- (183.8742pt, -526.5747pt)
         -- (183.8742pt, -526.5747pt)
         -- (183.8742pt, -526.5747pt) .. controls (188.5299pt, -526.5748pt) and (192.5038pt, -524.9287pt) .. (195.7959pt, -521.6366pt)
         -- (195.7959pt, -521.6366pt)
         -- (195.7959pt, -521.6366pt)
         -- (195.7959pt, -521.6366pt) .. controls (199.088pt, -518.3446pt) and (200.734pt, -514.3707pt) .. (200.734pt, -509.7149pt)
        ;
        \draw[black,semithick,line cap=round]
        (293.076pt, -500pt) .. controls (297.1815pt, -513pt) and (317.9786pt, -523pt) .. (343.0346pt, -523pt)
         -- (343.0346pt, -523pt)
         -- (343.0346pt, -523pt)
         -- (343.0346pt, -523pt) .. controls (368.0905pt, -523pt) and (388.8878pt, -513pt) .. (392.9932pt, -500pt)
        ;
        \draw[black,semithick,line cap=round]
        (385.2991pt, -509pt) .. controls (381.7531pt, -497pt) and (363.789pt, -488pt) .. (342.1463pt, -488pt)
         -- (342.1463pt, -488pt)
         -- (342.1463pt, -488pt)
         -- (342.1463pt, -488pt) .. controls (320.5036pt, -488pt) and (302.5397pt, -497pt) .. (298.9935pt, -509pt)
        ;
        \draw[black,semithick,line cap=round,xshift=158pt]
        (293.076pt, -500pt) .. controls (297.1815pt, -513pt) and (317.9786pt, -523pt) .. (343.0346pt, -523pt)
         -- (343.0346pt, -523pt)
         -- (343.0346pt, -523pt)
         -- (343.0346pt, -523pt) .. controls (368.0905pt, -523pt) and (388.8878pt, -513pt) .. (392.9932pt, -500pt)
        ;
        \draw[black,semithick,line cap=round,xshift=158pt]
        (385.2991pt, -509pt) .. controls (381.7531pt, -497pt) and (363.789pt, -488pt) .. (342.1463pt, -488pt)
         -- (342.1463pt, -488pt)
         -- (342.1463pt, -488pt)
         -- (342.1463pt, -488pt) .. controls (320.5036pt, -488pt) and (302.5397pt, -497pt) .. (298.9935pt, -509pt)
        ;
        \end{tikzpicture}
    \hfil 
    \begin{tikzpicture}[scale=0.27]
        \draw[black!30,semithick,line cap=round,line join=round]
        (28.7606pt, 198.6851pt) .. controls (67.41151pt, 212.0189pt) and (98.49816pt, 211.3599pt) .. (142.0935pt, 189.9503pt)
         -- (142.0935pt, 189.9503pt)
         -- (142.0935pt, 189.9503pt)
         -- (142.0935pt, 189.9503pt) .. controls (155.2504pt, 170.7468pt) and (140.0432pt, 138.0644pt) .. (116.3633pt, 146.1461pt)
         -- (116.3633pt, 146.1461pt)
         -- (116.3633pt, 146.1461pt)
         -- (116.3633pt, 146.1461pt) .. controls (110.4069pt, 148.1501pt) and (90.71743pt, 149.9833pt) .. (79.74713pt, 137.4308pt)
         -- (79.74713pt, 137.4308pt)
         -- (79.74713pt, 137.4308pt)
         -- (79.74713pt, 137.4308pt) .. controls (67.1635pt, 123.0323pt) and (61.87718pt, 108.1049pt) .. (53.43059pt, 98.54648pt)
        ;
        \draw[black,semithick,line cap=round,line join=round]
        (21.90163pt, 190.3999pt) .. controls (35.03999pt, 222.9408pt) and (57.06142pt, 151.6858pt) .. (57.06142pt, 117.2255pt)
         -- (57.06142pt, 117.2255pt)
         -- (57.06142pt, 117.2255pt)
         -- (57.06142pt, 117.2255pt) .. controls (57.06142pt, 82.7851pt) and (40.03458pt, 100.8768pt) .. (40.03458pt, 100.8768pt)
        ;
        \draw[black,semithick,line cap=round,line join=round]
        (186.0395pt, 168.3487pt) .. controls (202.7384pt, 186.7545pt) and (221.0504pt, 190.4473pt) .. (236.9873pt, 190.7012pt)
         -- (236.9873pt, 190.7012pt)
         -- (236.9873pt, 190.7012pt)
         -- (236.9873pt, 190.7012pt) .. controls (253.6828pt, 190.967pt) and (281.2499pt, 179.5671pt) .. (280.9072pt, 169.8004pt)
         -- (280.9072pt, 169.8004pt)
         -- (280.9072pt, 169.8004pt)
         -- (280.9072pt, 169.8004pt) .. controls (280.579pt, 160.4477pt) and (259.7332pt, 148.6423pt) .. (237.9705pt, 148.4162pt)
         -- (237.9705pt, 148.4162pt)
         -- (237.9705pt, 148.4162pt)
         -- (237.9705pt, 148.4162pt) .. controls (216.2743pt, 148.1909pt) and (205.8828pt, 156.6781pt) .. (186.0371pt, 168.3478pt)
         -- (186.0371pt, 168.3478pt)
         -- (186.0371pt, 168.3478pt)
         -- (186.0371pt, 168.3478pt) .. controls (168.5764pt, 178.6151pt) and (157.4824pt, 184.189pt) .. (142.092pt, 189.9447pt)
         -- (142.092pt, 189.9447pt)
         -- (142.092pt, 189.9447pt)
         -- (142.092pt, 189.9447pt) .. controls (125.1139pt, 197.105pt) and (100.6982pt, 163.5853pt) .. (116.3605pt, 146.152pt)
         -- (116.3605pt, 146.152pt)
         -- (116.3605pt, 146.152pt)
         -- (116.3605pt, 146.152pt) .. controls (130.951pt, 140.7547pt) and (139.9609pt, 129.2045pt) .. (151.544pt, 121.3358pt)
         -- (151.544pt, 121.3358pt)
         -- (151.544pt, 121.3358pt)
         -- (151.544pt, 121.3358pt) .. controls (179.8372pt, 101.6778pt) and (184.9466pt, 91.58969pt) .. (217.6929pt, 83.01947pt)
         -- (217.6929pt, 83.01947pt)
         -- (217.6929pt, 83.01947pt)
         -- (217.6929pt, 83.01947pt) .. controls (242.3145pt, 76.57556pt) and (263.8183pt, 79.76657pt) .. (281.5451pt, 87.32797pt)
         -- (281.5451pt, 87.32797pt)
         -- (281.5451pt, 87.32797pt)
         -- (281.5451pt, 87.32797pt) .. controls (294.7983pt, 92.98117pt) and (307.1575pt, 95.34445pt) .. (324.4563pt, 88.19006pt)
         -- (324.4563pt, 88.19006pt)
         -- (324.4563pt, 88.19006pt)
         -- (324.4563pt, 88.19006pt) .. controls (344.1301pt, 80.05347pt) and (364.5763pt, 70.68671pt) .. (380.2232pt, 70.68671pt)
         -- (380.2232pt, 70.68671pt)
         -- (380.2232pt, 70.68671pt)
         -- (380.2232pt, 70.68671pt) .. controls (401.2037pt, 70.68671pt) and (420.8594pt, 75.27075pt) .. (437.7501pt, 83.26691pt)
         -- (437.7501pt, 83.26691pt)
         -- (437.7501pt, 83.26691pt)
         -- (437.7501pt, 83.26691pt) .. controls (451.8103pt, 89.92303pt) and (468.0386pt, 89.92303pt) .. (482.0988pt, 83.26691pt)
         -- (482.0988pt, 83.26691pt)
         -- (482.0988pt, 83.26691pt)
         -- (482.0988pt, 83.26691pt) .. controls (498.9894pt, 75.27075pt) and (518.6452pt, 70.68671pt) .. (539.6256pt, 70.68671pt)
         -- (539.6256pt, 70.68671pt)
         -- (539.6256pt, 70.68671pt)
         -- (539.6256pt, 70.68671pt) .. controls (602.6642pt, 70.68671pt) and (653.7672pt, 112.05pt) .. (653.7672pt, 163.0735pt)
         -- (653.7672pt, 163.0735pt)
         -- (653.7672pt, 163.0735pt)
         -- (653.7672pt, 163.0735pt) .. controls (653.7672pt, 214.097pt) and (602.6642pt, 255.46pt) .. (539.6256pt, 255.46pt)
         -- (539.6256pt, 255.46pt)
         -- (539.6256pt, 255.46pt)
         -- (539.6256pt, 255.46pt) .. controls (518.6452pt, 255.46pt) and (498.9894pt, 250.8763pt) .. (482.0988pt, 242.8801pt)
         -- (482.0988pt, 242.8801pt)
         -- (482.0988pt, 242.8801pt)
         -- (482.0988pt, 242.8801pt) .. controls (468.0386pt, 236.224pt) and (451.8103pt, 236.224pt) .. (437.7501pt, 242.8801pt)
         -- (437.7501pt, 242.8801pt)
         -- (437.7501pt, 242.8801pt)
         -- (437.7501pt, 242.8801pt) .. controls (420.8595pt, 250.8762pt) and (401.2037pt, 255.46pt) .. (380.2233pt, 255.46pt)
         -- (380.2233pt, 255.46pt)
         -- (380.2233pt, 255.46pt)
         -- (380.2233pt, 255.46pt) .. controls (348.7039pt, 255.46pt) and (329.0266pt, 244.2595pt) .. (312.1046pt, 234.8669pt)
         -- (312.1046pt, 234.8669pt)
         -- (312.1046pt, 234.8669pt)
         -- (312.1046pt, 234.8669pt) .. controls (297.4287pt, 226.721pt) and (286.5509pt, 227.3481pt) .. (269.0003pt, 233.6318pt)
         -- (269.0003pt, 233.6318pt)
         -- (269.0003pt, 233.6318pt)
         -- (269.0003pt, 233.6318pt) .. controls (243.6133pt, 242.7213pt) and (213.7487pt, 249.2956pt) .. (186.274pt, 252.4861pt)
         -- (186.274pt, 252.4861pt)
         -- (186.274pt, 252.4861pt)
         -- (186.274pt, 252.4861pt) .. controls (133.5144pt, 258.631pt) and (29.91606pt, 262.9104pt) .. (8.379456pt, 196.4139pt)
         -- (8.379456pt, 196.4139pt)
         -- (8.379456pt, 196.4139pt)
         -- (8.379456pt, 196.4139pt) .. controls (-0.586074pt, 168.7211pt) and (4.307579pt, 130.651pt) .. (18.81081pt, 105.6681pt)
         -- (18.81081pt, 105.6681pt)
         -- (18.81081pt, 105.6681pt)
         -- (18.81081pt, 105.6681pt) .. controls (26.0347pt, 93.22784pt) and (40.06734pt, 79.12268pt) .. (53.84263pt, 74.47073pt)
         -- (53.84263pt, 74.47073pt)
         -- (53.84263pt, 74.47073pt)
         -- (53.84263pt, 74.47073pt) .. controls (84.0192pt, 64.30811pt) and (130.3151pt, 99.57724pt) .. (151.5332pt, 121.3243pt)
        ;
        \draw[red,ultra thick,line cap=round]
        (50.94615pt, 159.9069pt) .. controls (34.08578pt, 203.2091pt) and (107.8934pt, 237.8801pt) .. (179.302pt, 231.7581pt)
         -- (179.302pt, 231.7581pt)
         -- (179.302pt, 231.7581pt)
         -- (179.302pt, 231.7581pt) .. controls (325.2433pt, 219.2463pt) and (360.1235pt, 151.1pt) .. (274.3556pt, 122.8963pt)
         -- (274.3556pt, 122.8963pt)
         -- (274.3556pt, 122.8963pt)
         -- (274.3556pt, 122.8963pt) .. controls (204.4088pt, 99.8952pt) and (151.8974pt, 178.9907pt) .. (117.8708pt, 177.7642pt)
        ;
        \draw[red,ultra thick,dotted,line cap=round]
        (50.91499pt, 159.9053pt) .. controls (62.85999pt, 138.7637pt) and (88.99489pt, 177.24pt) .. (117.9864pt, 177.7458pt)
        ;
        \draw[black,semithick,line cap=round,xshift=39pt,yshift=679pt]
        (293.076pt, -500pt) .. controls (297.1815pt, -513pt) and (317.9786pt, -523pt) .. (343.0346pt, -523pt)
         -- (343.0346pt, -523pt)
         -- (343.0346pt, -523pt)
         -- (343.0346pt, -523pt) .. controls (368.0905pt, -523pt) and (388.8878pt, -513pt) .. (392.9932pt, -500pt)
        ;
        \draw[black,semithick,line cap=round,xshift=39pt,yshift=679pt]
        (385.2991pt, -509pt) .. controls (381.7531pt, -497pt) and (363.789pt, -488pt) .. (342.1463pt, -488pt)
         -- (342.1463pt, -488pt)
         -- (342.1463pt, -488pt)
         -- (342.1463pt, -488pt) .. controls (320.5036pt, -488pt) and (302.5397pt, -497pt) .. (298.9935pt, -509pt)
        ;
        \draw[black,semithick,line cap=round,xshift=197pt,yshift=679pt]
        (293.076pt, -500pt) .. controls (297.1815pt, -513pt) and (317.9786pt, -523pt) .. (343.0346pt, -523pt)
         -- (343.0346pt, -523pt)
         -- (343.0346pt, -523pt)
         -- (343.0346pt, -523pt) .. controls (368.0905pt, -523pt) and (388.8878pt, -513pt) .. (392.9932pt, -500pt)
        ;
        \draw[black,semithick,line cap=round,xshift=197pt,yshift=679pt]
        (385.2991pt, -509pt) .. controls (381.7531pt, -497pt) and (363.789pt, -488pt) .. (342.1463pt, -488pt)
         -- (342.1463pt, -488pt)
         -- (342.1463pt, -488pt)
         -- (342.1463pt, -488pt) .. controls (320.5036pt, -488pt) and (302.5397pt, -497pt) .. (298.9935pt, -509pt)
        ;
        \end{tikzpicture}
    \caption{Non-orientable surfaces of odd (left) and even (right) genus, with one-sided closed curves highlighted in thick red.}
    \label{fig:surfaces}
\end{figure}

Theorem~\ref{thm:dc_lovasz} implies that the Euler characteristic of $\Lo(G)$ is twice that of $S$. Moreover, the double cover of an orientable surface must also be orientable, which can be seen by considering the Jacobians of transition functions of an atlas.

Nakamoto \emph{et al.}~\cite{NNO04} defined an algebraic invariant of quadrangulations of surfaces called \emph{cycle parity}, denoted $\rho$, and used it to classify non-bipartite quadrangulations of closed non-orientable surfaces. They showed that $\rho$ falls into one of six types: $A$, $B$, $C$, $D$, $E$, or $F$~\cite[Theorems~6--8]{NNO04}. In particular, if the embedding of $G$ contains a Möbius strip, it contains a one-sided cycle (see Figure~\ref{fig:surfaces}). Consequently, $\Lo(G)$ is orientable if and only if all one-sided cycles are odd. Indeed, in the double cover of $G$, each one-sided cycle lifts either to two one-sided cycles (if even) or to a single two-sided cycle (if odd), which occurs precisely when $\rho_G$ is of type $A$, $B$, or $F$.

\begin{cor}\label{cor:class_L}
    Let $G \not\cong K_{2,3}$ be a connected non-bipartite quadrangulation of a surface $S$ in which every $4$-cycle is facial. 
    \begin{enumerate}
        \item If $S \cong \OS_k$, then $\Lo(G) \cong \OS_{2k-1}$.
        \item If $S \cong \NS_k$ and $G$ contains no even one-sided cycle, then $\Lo(G) \cong \OS_{k-1}$.
        \item If $S \cong \NS_k$ and $G$ contains an even one-sided cycle, then $\Lo(G) \cong \NS_{2k-2}$.
        
    \end{enumerate}
\end{cor}

We conclude the section by noting that the requirement that all $4$-cycles be facial is necessary for the Lovász complex to be a surface.

\begin{obs}\label{obs:not_surface}
    Let $G$ be a non-bipartite quadrangulation of a surface such that some $4$-cycle is not facial. Then $\Lo(G)$ is not a surface.
\end{obs}

\begin{proof}
    Let $a,b,c,d$ be the vertices of a cycle that is not a face in $G$. Then the edge $\{a,b\}$ is in at least three cycles in $G$, since it is in two faces. So the edge $\{\{a\},N(b)\}$ is in three triangles of $\Lo(G)$, with the third vertex of the triangle being the diagonal containing $a$ of each of the three $4$-cycles.  
\end{proof}

\section{Graphs whose Lovász complex is a surface} 

We now prove the ``backward'' direction of Theorem~\ref{thm:L-quad}.

\begin{thm}\label{thm:backward}
    Let $G$ be a $K_{2,3}$-free graph such that no neighborhood dominates another. If $\Lo(G)$ is a surface, then $G$ is a quadrangulation of $\Lo(G)/\Z2$ in which every $4$-cycle is facial.
\end{thm}

\begin{proof}
    Singletons and neighborhoods are in $V(\Lo(G))$ by Lemma~\ref{lem:all_vertices}. Therefore, any triangle in $\Lo(G)$ is of the form $\{\{u\}, A , N(v)\}$, where $u,v$ are vertices of $G$ and $u \in A \subset N(v)$. Without using the hypothesis that $G$ is $K_{2,3}$-free, we can already say that $A$ is contained in an even number of triangles. Indeed, when one considers the union of the triangles containing $A$, the edges $\{A,\{u\}\}$ and $\{A,N(v)\}$ incident to $A$ must alternate, as $N(v)$ cannot be a singleton.

    For $1 \leq i \leq j$, note $N(v_i)$ the neighborhoods containing $A$, and $\{u_i\}$ the singletons contained in $A$. Equivalently, $u_i$ are the elements of $A$ and $v_i$ the common neighbors of $A$. Since $\{u\} \subsetneq A$, $j \geq 2$. If $j \geq 3$, three elements of $A$ and two common neighbors would produce a $K_{2,3}$. So $A$ is contained in four triangles.

    Consider the subgraph $\widetilde G$ of $\Lo(G)$ induced by singletons and neighborhoods. This is a symmetric quadrangulation of $\Lo(G)$. No neighborhood dominates another, so the free  $\Z2$-action associating $\{v\}$ to $N(v)$ is well-defined, and the orbit space $\widetilde G/\Z2$ is isomorphic to $G$. Indeed, $\{\{v\},N(u)\} \in \Lo(G)$ if and only if $\{\{u\},N(v)\} \in \Lo(G)$. Since $\Z2$ is a finite group acting on the surface $\Lo(G)$, the orbit space $\Lo(G)/\Z2$ is also a surface.
\end{proof}

Note that if we drop the hypothesis of $G$ being $K_{2,3}$-free from Theorem~\ref{thm:backward}, then $G$ is embedded in $\Lo(G)$ with all faces bounded by an even cycle (not necessarily a $4$-cycle). 

The following is a straightforward corollary of Theorem~\ref{thm:backward} and the classification of surfaces.

\begin{cor}\label{cor:class_graphs}
    Let $G$ be a $K_{2,3}$-free graph such that no neighborhood dominates another.
    \begin{enumerate}
        \item If $\Lo(G) \cong \OS_{2k}$, then $G$ is a quadrangulation of $\NS_{2k+1}$ such that every $4$-cycle of $G$ bounds a face, and $G$ contains no even one-sided cycle.
        \item If $\Lo(G) \cong \NS_{2k}$, then $G$ is a quadrangulation of $\NS_{k+1}$ such that every $4$-cycle of $G$ bounds a face, and $G$ contains an even one-sided cycle if and only if $k$ is odd.
        \item If $\Lo(G) \cong \OS_{2k-1}$, then $G$ is a quadrangulation of either $\OS_k$ or $\NS_{2k}$ such that every $4$-cycle of $G$ bounds a face, and $G$ contains an even one-sided cycle.
    \end{enumerate}
\end{cor}

\section{Cohom-index of the Lovász complex of quadrangulations}\label{section5}

The goal of the last section is to prove Theorem~\ref{thm:archdeacon-ind}. Rather than working with the $\Z2$-index directly, we will prove a result on the cohom-index, and apply the inequality~\eqref{eq:inequalities}.

We adapt the proof in Archdeacon \emph{et al.}~\cite{AHNNO01}.
 
\begin{thm} \label{thm:cohom_quad}
    Let $G$ be a non-bipartite quadrangulation of a non-orientable surface $S$ in which every $4$-cycle is facial. Then $\cohom-ind(\Lo(G))= 2$ if $G$ is odd, otherwise $\cohom-ind(\Lo(G))= 1$.
\end{thm}

\begin{proof}
    Suppose $G$ has vertices $v_1, \ldots, v_n$. Let $\widetilde G$ be the subgraph of the $1$-skeleton of $\Lo(G)$ induced by the singletons and the neighborhoods. Note that $\widetilde G$ is a quadrangulation of $\Lo(G)$, and denote the set of its faces by $F$. Define a labelling $\lambda:V(\widetilde G) \to \{\pm 1, \ldots, \pm n\}$ by $\lambda(\{v_i\}) = i$ and $\lambda(N(v_i)) = -i$.

    Let $C$ be an orientizing $\ell$-cycle in $G$. Using the same idea as~\cite{AHNNO01}, define two orientations on couples of faces and edges (in that face); one from a fixed orientation of the orientable surface $S \setminus C$, the other from the labeling $\lambda$ by orienting $\{u,v\}$ from $u$ to $v$ if $|\lambda(u)| < |\lambda(v)|$. For each face $f \in F$, let $\varphi(f)$ be the difference between the number of edges in $f$ where the two orientations agree and the number where they disagree, as in~\cite{AHNNO01}. We obtain the congruence
    \begin{equation}\label{eq:cong1}
        \sum_{f \in F} \varphi(f) \equiv 2\ell \pmod 4.
    \end{equation}

    If the vertices in a face $f$ are not in cyclic order with respect to $|\lambda(\cdot)|$, then $\varphi(f)=0$, and if they are, then $\varphi(f) \in \{-2,2\}$. Therefore, if $r$ is the number of faces in $\widetilde G$ in cyclic order,
    \begin{equation}\label{eq:cong2}
    \sum_{f \in F} \varphi(G) \equiv 2r \pmod 4.
    \end{equation}
    It follows from \eqref{eq:cong1} and \eqref{eq:cong2} that $r$ and $\ell$ have the same parity.

    We will now triangulate $\widetilde G$ by dividing each face into two triangles (in a symmetric way). We wish to count the parity of the number of triangles whose middle vertex (by absolute value) has a different sign from the other two. Such triangles are called \emph{gray} in Matoušek~\cite[pp.\ 103--106]{Mat03}. It is not hard to verify that the only faces of $\widetilde G$ with exactly one gray triangle have vertices in cyclic order. Therefore, the number of gray triangles has the same parity as $r$, which has the same parity as $\ell$.

    We conclude that if $G$ is odd, then $\cohom-ind(\Lo(G)) = 2$, otherwise $\cohom-ind(\Lo(G)) = 1$.
\end{proof}

Theorem~\ref{thm:archdeacon-ind} is an easy consequence of Theorem~\ref{thm:cohom_quad}.
\begin{proof}[Proof of Theorem~\ref{thm:archdeacon-ind}]
    Let $G$ be an odd quadrangulation of a non-orientable surface $S$. If $G$ contains non-facial $4$-cycles, then $\Lo(G)$ is not a surface, as established in Observation~\ref{obs:not_surface}. To address this, consider the subcomplex $\mathsf K \subseteq \Lo(G)$ induced by the diagonals, singletons, and neighborhoods. This subcomplex $\mathsf K$ forms a double cover of $S$, and by an argument analogous to that in the proof of Theorem~\ref{thm:cohom_quad}, we find that $\cohom-ind(\mathsf K) = 2$. Applying the monotonicity of the index along with the inequality~\eqref{eq:inequalities}, we conclude that
    $\ind(\Lo(G)) \geq \ind(\mathsf K) \geq 2$.
\end{proof}

It is worth noting that while the level of precision provided by Nakamoto \emph{et al.}~\cite{NNO04} is not required to determine the homeomorphism class of the Lovász complex, it \emph{is} essential for determining the $\Z2$-index. In particular, two quadrangulations with homeomorphic Lovász complexes may have different $\Z2$-index or chromatic number.

It is shown in Gonçalves \emph{et al.}~\cite{GHZ10} that the cohom-index of a surface is equal to its $\Z2$-index. The following corollary follows. 

\begin{cor}
    Let $G$ be a non-bipartite quadrangulation of a non-orientable surface $S$ in which each $4$-cycle is facial. 
    \begin{enumerate}
        \item If $S \cong \NS_{2k+1}$ and $\rho_G$ is of type $A$ or $B$, then $\ind(\Lo(G))= 2$.
        \item If $S \cong \NS_{2k+1}$ and $\rho_G$ is of type $C$, then $\ind(\Lo(G))= 1$.
        \item If $S \cong \NS_{2k}$ and $\rho_G$ is of type $D$ or $E$, then $\ind(\Lo(G))= 1$.
        \item If $S \cong \NS_{2k}$ and $\rho_G$ is of type $F$, then $\ind(\Lo(G))= 2$.
    \end{enumerate}
\end{cor}

Note that Musin~\cite{Mus12} extended the Borsuk--Ulam theorem to manifolds other than the sphere; he called them \emph{BUT-manifolds}. The above corollary implies that $\Lo(G)$ with its canonical $\Z2$-action is a $2$-dimensional BUT-manifold if and only if $\rho_G$ is of type $A$, $B$, or $F$. 

\section{Acknowledgements}
The authors would like to express their gratitude to Christian Blanchet and Juan Paucar for helpful disussions. 

\bibliographystyle{plain}

\begin{thebibliography}{10}

\bibitem{AHNNO01}
D.~Archdeacon, J.~Hutchinson, A.~Nakamoto, S.~Negami, and K.~Ota.
\newblock Chromatic numbers of quadrangulations on closed surfaces.
\newblock {\em J. Graph Theory}, 37(2):100--114, 2001.

\bibitem{BM22}
C.~Blanchet and C.~Matmat.
\newblock The {B}orsuk--{U}lam theorem for 3-manifolds.
\newblock {\em Quaest. Math.}, 45(5):667--687, 2022.

\bibitem{BM08}
J.~A. Bondy and U.~S.~R. Murty.
\newblock {\em Graph theory}, volume 244 of {\em Graduate Texts in
  Mathematics}.
\newblock Springer, New York, 2008.

\bibitem{Cso05}
P.~Csorba.
\newblock {\em Non-tidy spaces and graph colorings}.
\newblock Ph{D} thesis, ETH Z\"{u}rich, 2005.

\bibitem{Cso08}
P.~Csorba.
\newblock On the simple {$\mathbb Z_2$}-homotopy types of graph complexes and
  their simple {$\mathbb Z_2$}-universality.
\newblock {\em Canad. Math. Bull.}, 51(4):535--544, 2008.

\bibitem{dL13}
M.~de~Longueville.
\newblock {\em A course in topological combinatorics}.
\newblock Universitext. Springer, New York, 2013.

\bibitem{GHZ10}
D.~L. Gon\c{c}alves, C.~Hayat, and P.~Zvengrowski.
\newblock The {B}orsuk--{U}lam theorem for manifolds, with applications to
  dimensions two and three.
\newblock In {\em Group actions and homogeneous spaces}, pages 9--28. Fak. Mat.
  Fyziky Inform. Univ. Komensk\'{e}ho, Bratislava, 2010.

\bibitem{Hat02}
A.~Hatcher.
\newblock {\em Algebraic topology}.
\newblock Cambridge University Press, Cambridge, 2002.

\bibitem{Hut95}
J.~P. Hutchinson.
\newblock Three-coloring graphs embedded on surfaces with all faces even-sided.
\newblock {\em J. Combin. Theory Ser. B}, 65(1):139--155, 1995.

\bibitem{Koz08}
D.~N. Kozlov.
\newblock {\em Combinatorial algebraic topology}, volume~21 of {\em Algorithms
  and Computation in Mathematics}.
\newblock Springer, Berlin, 2008.

\bibitem{Lov78}
L.~Lov{\'a}sz.
\newblock Kneser's conjecture, chromatic number, and homotopy.
\newblock {\em J. Combin. Theory Ser. A}, 25(3):319--324, 1978.

\bibitem{Mat03}
J.~Matou{\v{s}}ek.
\newblock {\em Using the {B}orsuk-{U}lam theorem}.
\newblock Universitext. Springer-Verlag, Berlin, 2003.

\bibitem{MZ04}
J.~Matou{\v{s}}ek and G.~M. Ziegler.
\newblock Topological lower bounds for the chromatic number: a hierarchy.
\newblock {\em Jahresber. Deutsch. Math.-Verein.}, 106(2):71--90, 2004.

\bibitem{MS02}
B.~Mohar and P.~D. Seymour.
\newblock Coloring locally bipartite graphs on surfaces.
\newblock {\em J. Combin. Theory Ser. B}, 84(2):301--310, 2002.

\bibitem{MT01}
B.~Mohar and C.~Thomassen.
\newblock {\em Graphs on surfaces}.
\newblock Johns Hopkins Studies in the Mathematical Sciences. Johns Hopkins
  University Press, Baltimore, MD, 2001.

\bibitem{Mus12}
O.~R. Musin.
\newblock Borsuk--{U}lam type theorems for manifolds.
\newblock {\em Proc. Amer. Math. Soc.}, 140(7):2551--2560, 2012.

\bibitem{NNO04}
A.~Nakamoto, S.~Negami, and K.~Ota.
\newblock Chromatic numbers and cycle parities of quadrangulations on
  nonorientable closed surfaces.
\newblock {\em Discrete Math.}, 285(1-3):211--218, 2004.

\bibitem{Tof95}
B.~Toft.
\newblock Colouring, stable sets and perfect graphs.
\newblock In {\em Handbook of combinatorics, {V}ol.\ 1,\ 2}, pages 233--288.
  Elsevier Sci. B. V., Amsterdam, 1995.

\bibitem{You96}
D.~A. Youngs.
\newblock {$4$}-chromatic projective graphs.
\newblock {\em J. Graph Theory}, 21(2):219--227, 1996.

\end{thebibliography}

\def\soft#1{\leavevmode\setbox0=\hbox{h}\dimen7=\ht0\advance \dimen7
  by-1ex\relax\if t#1\relax\rlap{\raise.6\dimen7
  \hbox{\kern.3ex\char'47}}#1\relax\else\if T#1\relax
  \rlap{\raise.5\dimen7\hbox{\kern1.3ex\char'47}}#1\relax \else\if
  d#1\relax\rlap{\raise.5\dimen7\hbox{\kern.9ex \char'47}}#1\relax\else\if
  D#1\relax\rlap{\raise.5\dimen7 \hbox{\kern1.4ex\char'47}}#1\relax\else\if
  l#1\relax \rlap{\raise.5\dimen7\hbox{\kern.4ex\char'47}}#1\relax \else\if
  L#1\relax\rlap{\raise.5\dimen7\hbox{\kern.7ex
  \char'47}}#1\relax\else\message{accent \string\soft \space #1 not
  defined!}#1\relax\fi\fi\fi\fi\fi\fi}

\end{document}